\documentclass{article}
\usepackage{amsmath}
\usepackage{amssymb}
\usepackage{graphicx}
\usepackage{epstopdf}
\usepackage{amsthm}
\usepackage{hyperref}
\usepackage{amsfonts}
\usepackage{array}
\usepackage{amsthm}
\usepackage{tabularx}
\usepackage{cite}
\newtheorem{theorem}{Theorem}
\newtheorem{remark}{Remark}

\newtheorem{lemma}{Lemma}
\newtheorem{corollary}{Corollary}
\newtheorem{example}{Example}
\begin{document}
\title{Monotone function estimator and its application}
\author{Yunyi Zhang Dimitris N. Politis Jiazheng Liu and Zexin Pan}
\maketitle
\abstract{In this paper, the model $Y_i=g(Z_i),\ i=1,2,...,n$ with $Z_i$ being random variables with known distribution and $g(x)$ being unknown strictly increasing function is proposed and almost sure convergence of estimator for $g(x)$ is proved for i.i.d and short range dependent data. Confidence intervals and bands are constructed for i.i.d data theoretically and confidence intervals are introduced for short range dependent data through resampling. Besides, a test for equivalence of $g(x)$ to the desired function is proposed. Finite sample analysis and application of this model on an urban waste water treatment plant's data is demonstrated as well.}
\section{Introduction and assumptions}
\subsection{Introduction}
In this article, we focus on model
\begin{equation}
Y_i=g(Z_i), i=1,2,...,n
\label{Equation_He}
\end{equation}
and we try to estimate strictly increasing function $g(x)$ (we call it transfer function) for given $x$ and random variable $Z_i,\ i=1,2,..,n$ whose distribution are known under some constraints. We first provide some examples to clarify the motivation to estimate transfer function $g(x)$.
\begin{example}
\label{Exam_1}
Suppose there is a production line and we want to control the quality of products and minimize the cost of materials at the same time. It is reasonable to assume that the quality of products, $Y$ is an decreasing function $g$ of property of materials, $\vert Z-z_0\vert$ with $z_0$ being the design point. Moreover, the distribution of quality of materials can assume to be known. (For example, tensile strength of materials satisfies Weibull distribution \cite{Trustrum1979}.) However, it is difficult to use regression model since testing materials' quality is of great cost and always brings damage to materials. Instead, if the distribution of quality of materials is known, then distribution of $\vert Z-z_0\vert$ can be calculated and model \ref{Equation_He} can be applied. After estimating $g$, we know how sensitive the quality of materials makes influence on products.
\end{example}
\begin{example}
\label{Exam_2}
Consider the model in figure \ref{Amplif}. Suppose the probability distribution of input signal is known and the output signal data can be acquired. Then, two things are worth considering. The first one is to understand how the amplifier enlarges the input signal, that is, to estimate the transfer function $g(x)$. The second thing is to test whether the transfer function $g$ coincides with the expected transfer function $h$, which comes from physical laws or experience. For example, according to \cite{Second}, measured concentration $Y_i,\ i=1,2,...$ can be modelled as
\begin{equation}
Y_i=\mu\exp(a+b\times Z_i),\ i=1,2,...
\end{equation}
With $\mu$ the true concentration and $a,b$ unknown constant, $Z_i$ being standard normal random variable (but cannot be observed in measurement). Then, researchers having these data may be willing to justify the correctness of this model.
\begin{figure}
  \centering
  \includegraphics[width=9cm]{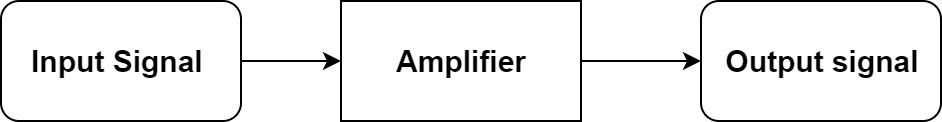}\\
  \caption{A standard amplifier system}\label{Amplif}
\end{figure}
\end{example}
\begin{example}
\label{Exam_3}
In the third example, we consider a type of time series data
\begin{equation}
\begin{aligned}
Z_n=\Sigma_{k=1}^{\infty}a_k\epsilon_{n-k},\ \epsilon\sim i.i.d\\
Y_n=g(Z_n)
\end{aligned}
\end{equation}
Here $Z_n$ is known and we want to estimate $g(x)$ for some $x$. For example, in \cite{doi:10.1002/0470011815.b2a12003}, daily number of respiratory symptoms per child is recorded and is related to daily $SO_2$ and $NO_2$. In that paper, transfer function $g(x)=v_0\log(x)+\epsilon$ with $\epsilon$ being an ARIMA series are considered. If instead, we ignore the errors $\epsilon$ and want to estimate transfer function in a non-parametric way, then model \ref{Equation_He} can be applied to this problem.
\end{example}
To summarize, example \ref{Exam_1} and \ref{Exam_3} involves estimating transfer function $g(x)$ in i.i.d data and dependent data, and example \ref{Exam_2} involves testing equivalence of transfer function. All of these three topics will be covered in this paper.

According to \cite{10.2307/2236692}, suppose $Z$ is a random variable with cumulative distribution $F_Z$, then $F_Z(Z)$ is of uniform distribution. Thus, random variables $Z$ with strictly increasing cumulative distribution function can be naturally related to a random variable $U$ with uniform distribution by choosing $g(x)=F_Z^{-1}(x)$. There is lots of discussion on estimating $F_Z(x)$ and $F_Z^{-1}(x)$, and related results can be found in \cite{10.2307/2958864} and \cite{wu2005}. However, there are few papers discussing model \ref{Equation_He}. We suppose the transfer function $g(x)$ in model \ref{Equation_He} is strictly increasing and random variables $Z_i,\ i=1,2,...$, have increasing distribution function in this paper. Under some constraints, we provide a method to estimate $g(x)$ and to construct point-wise confidence interval for i.i.d data and short-range dependent data. We also discuss how to construct confidence bands for i.i.d data and provide a goodness of fit test on equivalence of $g(x)$ and the expected transfer function $h(x)$.

In section 2, we will demonstrate how to estimate transfer function and construct confidence intervals and bands for i.i.d data. We also provide a test similar to Kolmogorov-Smironv test\cite{Testing_Statistical} on testing whether $g(x)=h(x)$, the expected function. In section 3, we discuss how to estimate transfer function and how to construct confidence interval through sub-sampling methods for short-range dependent data. In section 4, several numerical examples are provided and conclusion is made in the last section.
\subsection{Frequently used notations and assumptions}
In this part, we introduce frequently used notations for this article, other symbols will be defined when being used. We will also list basic assumptions and constraints on random variables and transfer function below.

Suppose that $Z_i,\ i=1,2,...,n$ are random variables with known cumulative distribution function $F_Z(x)$ and density $f_Z(x)$, $Y_i$ being unknown random variables satisfying $Y_i=g(Z_i)\ \forall i$. We define empirical distribution function as
\begin{equation}
F_n^K(x)=\frac{1}{n}\Sigma_{i=1}^n\mathbf{1}_{K_i\leq x},K=Y,Z
\end{equation}
Quantile and sample quantile function as \begin{equation} \xi^K(p)=\inf\left\{x|F_K(x)\geq p\right\},\ \xi_n^K(p)=\inf\left\{x|F_n^K(x)\geq p\right\} \end{equation}

Assumption A1: $Z_i$ are i.i.d with strictly increasing cumulative distribution function.

Assumption A2: $Z_i$ are causal stationary linear short range dependent processes (details can be seen in \cite{wu2005}). That is, $Z_k=\Sigma_{i=0}^{\infty}a_i\epsilon_{k-i}$ with $\epsilon_i$ being i.i.d. random variables and satisfy
\begin{equation}
\sup_{x\in\mathbf{R}}f_{\epsilon}(x)+\vert f^{'}_\epsilon(x)\vert<\infty
\end{equation}
Here $f_\epsilon$ is density of innovation $\epsilon$. Moreover, suppose $\exists \alpha>0$, $q>2$ such that $\mathbf{E}\vert\epsilon_k\vert^\alpha<\infty$ and
\begin{equation}
\Sigma_{i=n}^\infty\vert a_i\vert^{\min(\alpha/q,1)}=O((\log n)^{-1/q})
\end{equation}

Assumption A3: $Z_i$ satisfies $\alpha-$ mixing condition (Details can be seen at \cite{subsampling} and \cite{Strong_Mixing})

Assumption B1: $g$ is strictly monotonic increasing (but does not have to be positive, so for decreasing $g$, $h=-g$ is increasing)

Assumption B2: $g$ is differentiable

Assumption B3: $g$ is twice continuous differentiable, $f_Z$ is continuous differentiable on $(a,b)$ defined in table \ref{Notation_sym}. Moreover, we assume that $\exists \gamma>0$ such that
\begin{equation}
\sup_{a<x<b} \frac{F_Z(x)(1-F_Z(x))}{f_Z^2(x)}\vert f_Z^{'}(x)-f_Z(x)\frac{g^{''}(x)}{g^{'}(x)}\vert\leq \gamma
\end{equation}
Notice that this equation implies that $g^{'}(x), f_Z(x)>0$ on $(a,b)$, correspondingly $f_Y(x)>0,\ x\in[g(a),g(b)]$.
\begin{table}
\caption{Frequently used notations}
\begin{tabular}{l l}
\hline\hline
Notation & Meaning\\
\hline
$F_Z(x),\ f_Z(x)$ & Cumulative distribution and density of known random variable $Z$\\
\hline
 $F_Y(x),\ f_Y(x)$ & Cumulative distribution and density of unknown random variable $Y$\\
 \hline $f_n^Y(x)$ & Estimated density of unknown random variable\\
 \hline $\xi^K(p), K=Y,Z$ & $p$th quantile function of distribution $F_K$\\
 \hline $\xi_n^K(p), K=Y,Z$ & $p$th sample quantile function of random variable $K$ \\
 \hline $g(x)$ & Transfer function satisfying $Y_i=g(Z_i),\ i=1,2,...n$\\
 \hline $\widehat{g}(x)$ & Estimated transfer function at x\\
 \hline $F_n^K(x)$ & Empirical distribution function of random variable $K=Y,Z$\\
 \hline $a,b$ & Here, $-\infty\leq a=\sup\left\{x|F_Z(x)=0\right\}$, $\infty\geq b=\inf\left\{x|F_Z(x)=1\right\}$\\
 \hline $\mathbf{1}_{K\in A},K=Y,Z$ & If $K\in A$, then function is equal to 1 and 0 otherwise\\
 \hline\hline
 \end{tabular}
 \label{Notation_sym}
 \end{table}
\section{Estimation of transfer function with i.i.d data}
In this section, we discuss estimation and test of transfer function on i.i.d data, including estimation, construction of confidence intervals and confidence bands. Based on Kolmogorov-Smirnov test, we provide a test on whether the transfer function is equal to the desired one and discuss performance of test under an alternative. First we provide two lemma.
\begin{lemma}
\label{lemma_trans}
Assume random variable $Y,Z$ satisfy $Y=g(Z)$ and g satisfies B1, with the notation in table \ref{Notation_sym}, then we have
\begin{equation}
F_Y(g(x))=F_Z(x),\forall x\in[a,b]
\end{equation}
\end{lemma}
\begin{proof}
Because $g$ is strictly increasing, we have $F_Y(g(x))=P(Y\leq g(x))=P(g(Z)\leq g(x))=P(Z\leq x)=F_Z(x)$ and the lemma is proved
\end{proof}
\begin{lemma}
\label{lm2} Assume B1, random variable $Y=g(Z)$, then we have, $\forall p\in(0,1)$
\begin{equation} \xi^Y(p)=g(\xi^Z(p)),\ \xi_n^Y(p)=g(\xi_n^Z(p))
\end{equation}
\end{lemma}
\begin{proof}
From definition, on one hand, $F_Y(g(\xi^Z(p)))=F_Z(\xi^Z(p))\geq p$, this is because $F_Z$ is right continuous. Therefore, $\xi^Y(p)\leq g(\xi^Z(p))$. On the other hand, since $g$ is strictly increasing, its inverse function $g^{-1}(y)$ is strictly increasing. Therefore we have $\xi^Z(p)\leq g^{-1}(\xi^Y(p))\Rightarrow g(\xi^Z(p))\leq \xi^Y(p)$, and the first part is proved. For the second part, we notice that $F_n^K(x), K=Y,Z$ are also a right continuous cumulative distribution functions and thus the discussion above can be directly applied to $\xi_n^Y(p),\xi_n^Z(p)$, and the second part is proved.
\end{proof}
We now start estimation of transfer function $g(x)$.
\subsection{Estimation of transfer function}
\begin{theorem}
Suppose A1 and B1, and for $\forall x\in(a,b)$ being given, define $\widehat{g}(x)=\xi_n^Y(F_Z(x))$. Then we have
\begin{equation}
\widehat{g}(x)\to_{a.s.} g(x),\ n\to\infty
\end{equation}
Moreover, for $\alpha\in(0,1/2)$ being given, we suppose $\zeta(y)$ being quantile function of standard normal distribution, then we have \begin{equation}
\lim\inf_{n\to\infty} P(\xi_n^Y(c_1)\leq g(x)\leq \xi_n^Y(c_2))\geq 1-\alpha
\end{equation}
Here, $c_1=F_Z(x)+\frac{\zeta(\alpha/2)\sqrt{F_Z(x)(1-F_Z(x))}}{\sqrt{n}}$, $c_2=F_Z(x)+\frac{\zeta(1-\alpha/2)\sqrt{F_Z(x)(1-F_Z(x))}}{\sqrt{n}}$ \end{theorem}
\begin{proof}
For the 1st part, according to \cite{Probability_Stochastic},
\begin{equation} \widehat{g}(x)\to_{a.s.}g(x)\Leftrightarrow\Sigma_{n=1}^{\infty}\mathbf{1}_{\vert \widehat{g}(x)-g(x)\vert> \epsilon}<\infty\ \forall\epsilon>0
\end{equation} Since
\begin{equation}
\mathbf{1}_{\widehat{g}(x)-g(x)>\epsilon}\leq \mathbf{1}_{F_n^Y(\xi^Y(F_Z(x))+\epsilon)\leq F_Z(x)}
\end{equation}
From strong law of large number, we have $F_n^Y(\xi^Y(F_Z(x))+\epsilon)\to_{a.s.}F^Y(\xi^Y(F_Z(x))+\epsilon)>F_Z(x)$, thus
\begin{equation}
\Sigma_{i=1}^\infty \mathbf{1}_{F_n^Y(\xi^Y(F_Z(x))+\epsilon)\leq F_Z(x)}<\infty
\end{equation}
Also, similarly we can get that $\Sigma_{n=1}^\infty \mathbf{1}_{\widehat{g}(x)-g(x)<-\epsilon}<\infty$ and we prove the result.

For the second part, we prove that
\begin{equation}
\lim\sup P(\xi_n^Y(c_1)>g(x))\leq\alpha/2,\ \lim\sup P(\xi_n^Y(c_2)<g(x))\leq\alpha/2
\end{equation}
For $\xi_n^Y(c_1)>g(x)\Rightarrow F_n^Y(g(x))<c_1$, from central limit theorem, and lemma \ref{lemma_trans}, we have
\begin{equation}
\sqrt{n}(F_n^Y(g(x))-F_Z(x))\to_{D}N(0,F_Z(x)(1-F_Z(x)))
\end{equation} Thus,
\begin{equation}
\lim\sup P(\xi_n^Y(c_1)>g(x))\leq \lim P\left(\frac{\sqrt{n}(F_n^Y(g(x))-F_Z(x))}{\sqrt{F_Z(x)(1-F_Z(x))}}\leq \zeta(\alpha/2)\right)=\alpha/2 \end{equation}
Similarly, we have $\lim\sup P(\xi_n^Y(c_2)<g(x))\leq \alpha/2$ and the theorem is proved. Since $x$ is constraint, $c_2-c_1=O(1/\sqrt{n})$.

This result can be applied to construct point-wise confidence interval.
\end{proof}
We consider construction of confidence band in the next theorem.
\begin{theorem}
\label{The_2}
Suppose A1, B1, B3, and suppose $\delta_n=(25\log\log n)/n$, define $\phi(x)$ as a kernel function satisfying the following condition: 1) $\phi$ is of finite support, i.e. there exists a compact interval $[d_1,d_2]$ such that $supp\ \phi\subseteq[d_1,d_2]$. 2) $\phi$ is continuous differentiable on $[d_1,d_2]$. 3) $\int_{d_1}^{d_2}\phi(x)=1$. We define the estimated density $f_n^Y(x)$ as
\begin{equation}
f_n^Y(x)=\frac{1}{nh}\Sigma_{i=1}^{n}\phi(\frac{x-Y_i}{h})
\end{equation}
Here, $h=h(n)$ is a bandwidth satisfying $(\log\log n)^{1/2}h\to 0$ and $\frac{\sqrt{n}h^2}{\log\log n}\to\infty$. Also suppose that $[c,d]$ is a closed interval in $\mathbf{R}$ such that $a<c<d<b$. Then we can define a Kiefer process $K(y,n),0\leq y\leq 1$\cite{Strong_Approximation_Quantile} such that \begin{equation}
\sup_{c\leq x\leq d}\vert\sqrt{n}(\widehat{g}(x)-g(x))f_n^Y(\widehat{g}(x))-\frac{K(F_Z(x),n)}{\sqrt{n}}\vert\to 0\ a.s.
\end{equation}
\end{theorem}
\begin{proof}
Because of B3, then according to \cite{Strong_Approximation_Quantile}, since $Y=g(Z), Z\in[a,b]$, and $g$ strictly increasing, then $Y\in[g(a),g(b)]$ and according to lemma \ref{lemma_trans}, we have $f_Y(g(x))g^{'}(x)=f_Z(x)$, $f_Y^{'}(g(x))g^{'}(x)^2+f_Y(g(x))g^{''}(x)=f_Z^{'}(x)$, thus suppose $z=g(x)$ and \begin{equation} \begin{aligned} \sup_{g(a)< z< g(b)}F_Y(z)(1-F_Y(z))\vert\frac{f_Y^{'}(z)}{f_Y^{2}(z)}\vert\\ =\sup_{a<x<b}\frac{F_Z(x)(1-F_Z(x))}{f_Z^2(x)}\vert f_Z^{'}(x)-f_Z(x)\frac{g^{''}(x)}{g^{'}(x)}\vert\leq\gamma \end{aligned} \end{equation} There exists a version of Kiefer process, such that \begin{equation} \sup_{\delta_n\leq F_Z(x)\leq 1-\delta_n}\vert n(\widehat{g}(x)-g(x))f_Y(g(x))-K(F_Z(x),n)\vert=_{a.s.}O((n\log\log n)^{1/4}(\log n)^{1/2}) \label{delta_n} \end{equation} For sufficiently large $n$, $\delta_n<F_Z(c)<F_Z(d)<1-\delta_n$ and the estimation above holds for $\forall x\in[c,d]$. On the other hand, for $\phi^{'}$ is continuous on its support $[d_1,d_2]$ and equal to 0 outside its support, define $\phi_m=\max_{x\in[d_1,d_2]}\vert\phi^{'}\vert$, from mean value theorem, we have, $\exists \eta\in \mathbf{R}$ such that \begin{equation} \vert f_n^Y(\widehat{g}(x))-f_n^Y(g(x))\vert=\vert f_n^Y(\eta)^{'}(\widehat{g}(x)-g(x)\vert\leq\frac{\phi_m}{h^2}\vert\widehat{g}(x)-g(x)\vert \end{equation} We next consider $f_n^Y(g(x))-f_Y(g(x))$. From integral transformation, we have \begin{equation} f_n^Y(g(x))-f_Y(g(x))=\frac{1}{h}\int_{d_1}^{d_2}F_n^Y(g(x)-hy)\phi^{'}(y)dy-f_Y(g(x)) \label{Prob_Int} \end{equation} From theorem A in \cite{Strong_Approximation_Quantile}, since $F_Y(Y_i)$ are uniform random variable, we pick $y=F_Y(g(x)-hz)$ in that theorem, suppose that $c\leq x\leq d$ and $h$ sufficiently small such that $g(a)<g(x)-hz,\ g(b)>g(x)-hz$ use lemma \ref{lemma_trans} and we have \begin{equation} \sup_{c\leq x\leq d}\vert n(F_n^Y(g(x)-hz)-F_Y(g(x)-hz))-K(F_Z(x),n)\vert=_{a.s.}O(\log^2 n) \end{equation} Therefore, for $n$ sufficiently large and $h<h_0$, $h_0$ sufficiently small, we have $g(a)<g(c)-h_0d_2<g(d)-h_0d_1<g(b)$ and since $f_Y(x)$ is continuous differentiable according to B3, its derivative  at $[g(c)-h_0d_2,g(d)-h_0d_1]$ is bounded, suppose $f_0=\max_{x\in [g(c)-h_0d_2,g(d)-h_0d_1]}\vert f_Y^{'}(x)\vert$. Therefore, equation \ref{Prob_Int} is equivalent as \begin{equation} \frac{1}{h}\int_{d_1}^{d_2}(F_n^Y(g(x)-hy)-F_Y(g(x)-hy))\phi^{'}(y)dy+\int_{d_1}^{d_2}\phi(y)(f_Y(g(x)-hy)-f_Y(g(x)))dy \end{equation} Moreover, from the law of iterated logarithm \cite{Strong_Approximation_Quantile}, we have \begin{equation} \lim\sup_{n\to\infty}\sup_{0\leq y\leq 1}\vert K(y,n)\vert/(2n\log\log n)^{1/2}=_{a.s.}1/2 \end{equation} Thus, from equation \ref{delta_n}, for sufficiently large $n$, \begin{equation} \sup_{x\in[c,d]}\vert \widehat{g}(x)-g(x)\vert\leq \sup_{x\in[c,d]}\vert \frac{K(F_Z(x),n)}{n\vert f_Y(g(x))\vert}\vert+O_{a.s.}(\frac{(\log\log n)^{1/4}(\log n)^{1/2}}{n^{3/4}})=O_{a.s.}(\frac{(\log\log n)^{1/2}}{\sqrt{n}}) \end{equation} since from assumption, $[c,d]$ is a closed interval and $\min_{x\in[c,d]}\vert f_Y(g(x))\vert>0$. Besides, we also have \begin{equation} \begin{aligned} \sup_{c\leq x\leq d}\vert f_n^Y(g(x))-f_Y(g(x))\vert\leq (\sup_{0\leq y\leq 1}\vert \frac{K(y,n)}{n}\vert+O(\frac{\log^2 n}{n}))\frac{\phi_m(d_2-d_1)}{h}+f_0h\int_{d_1}^{d_2}\vert y\phi(y)\vert dy\\ =_{a.s.}O(h+\frac{(\log\log n)^{1/2}}{h\sqrt{n}}) \end{aligned} \end{equation} To prove that theorem, from triangle inequality, \begin{equation} \begin{aligned} \sup_{c\leq x\leq d}\vert\sqrt{n}(\widehat{g}(x)-g(x))f_n^Y(\widehat{g}(x))-\frac{K(F_Z(x),n)}{\sqrt{n}}\vert\\ \leq \sup_{c\leq x\leq d}\vert\sqrt{n}(\widehat{g}(x)-g(x))f_Y(g(x))-\frac{K(F_Z(x),n)}{\sqrt{n}}\vert\\ +\sup_{c\leq x\leq d}\sqrt{n}\vert\widehat{g}(x)-g(x)\vert\vert f_n^Y(\widehat{g}(x))-f_n^Y(g(x))\vert\\ +\sup_{c\leq x\leq d}\sqrt{n}\vert\widehat{g}(x)-g(x)\vert\vert f_n^Y(g(x))-f_Y(g(x))\vert\\ \leq O_{a.s.}(\frac{(\log\log n)^{1/4}(\log n)^{1/2}}{n^{1/4}})+\sup_{c\leq x\leq d}\frac{\sqrt{n}\phi_m}{h^2}\vert\widehat{g}(x)-g(x)\vert^2+O_{a.s.}((\log\log n)^{1/2}h+\frac{(\log\log n)}{h\sqrt{n}})\\ =O_{a.s.}(\frac{(\log\log n)^{1/4}(\log n)^{1/2}}{n^{1/4}})+O_{a.s.}(\frac{\log\log n}{\sqrt{n}h^2})+O_{a.s.}((\log\log n)^{1/2}h+\frac{(\log\log n)}{h\sqrt{n}}) \end{aligned} \end{equation} Thus, let $(\log\log n)^{1/2}h\to 0$ and $\frac{\sqrt{n}h^2}{\log\log n}\to\infty$, we prove the result. \end{proof} \begin{remark} $\phi$ and $h(n)$ being defined on the theorem exists. For example, we can let \begin{equation} \phi(x)= \begin{cases} \frac{1}{2\pi}(1+\cos(x))\ x\in[-\pi,\pi]\\ 0\ others \end{cases} \end{equation} and let $h(n)=(1/n)^{1/6}$satisfies condition. \end{remark} \begin{corollary}[Confidence band within an interval] Suppose the same conditions in theorem \ref{The_2}, and suppose $c>0$ is a positive number, then we have \begin{equation} \lim_{n\to\infty}\sup P(\sup_{c\leq x\leq d}\vert\sqrt{n}(\widehat{g}(x)-g(x))f_n^Y(\widehat{g}(x))\vert>c)\leq P(\sup_{0\leq y\leq 1}\vert B(y)\vert>c)=\Sigma_{k\neq 0}(-1)^{k+1}\exp(-2k^2c^2) \end{equation} \end{corollary}
\begin{proof}
Define $A_n=\left\{\sup_{c\leq x\leq d}\vert\sqrt{n}(\widehat{g}(x)-g(x))f_n^Y(\widehat{g}(x))\vert>c\right\}$, according to theorem \ref{The_2}, for \begin{equation}
\sup_{c\leq x\leq d}\vert\sqrt{n}(\widehat{g}(x)-g(x))f_n^Y(\widehat{g}(x))\vert\leq \sup_{0\leq y\leq 1}\vert\frac{K(y,n)}{\sqrt{n}}\vert+\sup_{c\leq x\leq d}\vert\sqrt{n}(\widehat{g}(x)-g(x))f_n^Y(\widehat{g}(x))-\frac{K(F_Z(x),n)}{\sqrt{n}}\vert \end{equation}
 And for $\forall\ \epsilon>0$ given, for sufficiently large $n$, $P(\sup_{c\leq x\leq d}\vert\sqrt{n}(\widehat{g}(x)-g(x))f_n^Y(\widehat{g}(x))-\frac{K(F_Z(x),n)}{\sqrt{n}}\vert<\epsilon)=1$. Therefore, $P(A_n)\leq P(\sup_{0\leq y\leq 1}\vert\frac{K(y,n)}{\sqrt{n}}\vert>c-\epsilon)$ for large $n$. According to \cite{Quantile_Process}, $K(y,n)/\sqrt{n}$ is a Brownian bridge and according to \cite{Quantile_Process_2}, $P(\sup_{0\leq y\leq 1}\vert B(y)\vert\leq c)=1-\Sigma_{k\neq 0}(-1)^{k+1}\exp(-2k^2c^2)$. Thus, from continuity of measure, \begin{equation}
 \lim\sup P(A_n)\leq\lim_{\epsilon\to 0}P(\sup_{0\leq y\leq 1}\vert\frac{K(y,n)}{\sqrt{n}}\vert>c-\epsilon)=P(\sup_{0\leq y\leq 1}\vert\frac{K(y,n)}{\sqrt{n}}\vert\geq c)
 \end{equation}
 The final thing is to prove that $P(\sup_{0\leq y\leq 1}\vert\frac{K(y,n)}{\sqrt{n}}\vert=c)=0$. From continuity of measure, for $c>0$, \begin{equation} \begin{aligned} P(\sup_{0\leq y\leq 1}\vert\frac{K(y,n)}{\sqrt{n}}\vert=c)=\lim_{n\to\infty} P(c-1/n<\sup_{0\leq y\leq 1}\vert\frac{K(y,n)}{\sqrt{n}}\vert\leq c+1/n)\\ \leq \lim_{n\to\infty}\Sigma_{k\neq 0}\exp(-2k^2(c-1/n)^2)-\exp(-2k^2(c+1/n)^2) \end{aligned} \end{equation} From mean value theorem, there exists $\eta_k\in[c-1/k,c+1/k]$, such that $\exp(-2k^2(c-1/n)^2)-\exp(-2k^2(c+1/n)^2)=-8\exp(-2k^2\eta_k^2)k^2\eta_k/n$, so $\Sigma_{k\neq 0}\exp(-2k^2(c-1/n)^2)-\exp(-2k^2(c+1/n)^2)=O(1/n)$ thus the result is proved and we can use this observation to construct confidence band in a closed interval.
 \label{observation}
 \end{proof}
 \subsection{Testing}
 In this section, we mainly consider testing $H_0: g=h$ versuses $H_1: g\neq h$. Here $h$ is a known or desired transfer function and $g$ is the underlying one. We consider the test
 \begin{equation}
 \sup_{\delta_n\leq F_Z(x)\leq 1-\delta_n}\sqrt{n}\frac{f_Z(x)}{h^{'}(x)}\vert \widehat{g}(x)-h(x)\vert\leq c
 \end{equation}
 for accepting $H_0$. Here $c$ is a positive constant and $\delta_n$ is the same as in theorem \ref{The_2}. We will discuss its behavior under the null and an alternative.

 \begin{theorem}
 Suppose A1, B1, B3. Consider testing $H_0: g(x)=h(x)\ \forall\ x\in(a,b)$ versus $H_1: \exists a<x<b$ such that $g(x)\neq h(x)$. Suppose $\delta_n$ is defined the same as in theorem \ref{The_2}. Then under the null hypothesis, we have, given $c>0$,
 \label{Test_Null}
 \begin{equation}
 P(\sup_{\delta_n\leq F_Z(x)\leq 1-\delta_n}\sqrt{n}\frac{f_Z(x)}{h^{'}(x)}\vert \widehat{g}(x)-h(x)\vert>c)\to P(\sup_{0\leq y\leq 1}\vert B(y)\vert>c)=\Sigma_{k\neq 0}(-1)^{k+1}\exp(-2k^2c^2)
 \end{equation}
 Here $B(y)$ is a Brownian bridge.
 \end{theorem}
 \begin{proof} According to equation \ref{delta_n}, define events $M_n=\sup_{\delta_n\leq F_Z(x)\leq 1-\delta_n}\sqrt{n}\frac{f_Z(x)}{h^{'}(x)}\vert \widehat{g}(x)-h(x)\vert>c$. Since
 \begin{equation}
 \sup_{\delta_n\leq F_Z(x)\leq 1-\delta_n}\sqrt{n}\frac{f_Z(x)}{h^{'}(x)}\vert \widehat{g}(x)-h(x)\vert\leq \sup_{0\leq y\leq 1}\vert\frac{K(y,n)}{\sqrt{n}}\vert+\sup_{\delta_n\leq F_Z(x)\leq 1-\delta_n}\vert\sqrt{n}\frac{f_Z(x)}{h^{'}(x)} (\widehat{g}(x)-h(x))-\frac{K(F_Z(x),n)}{\sqrt{n}}\vert
 \end{equation}
 Thus, for $\forall\ \epsilon>0$, for sufficiently large $n$,
 \begin{equation}
 P(\sup_{\delta_n\leq F_Z(x)\leq 1-\delta_n}\vert\sqrt{n}\frac{f_Z(x)}{h^{'}(x)} (\widehat{g}(x)-h(x))-\frac{K(F_Z(x),n)}{\sqrt{n}}\vert<\epsilon)=1 \end{equation}
 and
 \begin{equation}
 \lim \sup P(M_n)\leq \lim\sup P(\sup_{0\leq y\leq 1}\vert\frac{K(y,n)}{\sqrt{n}}\vert>c-\epsilon)=P(\sup_{0\leq y\leq 1}\vert B(y)\vert>c-\epsilon) \end{equation}
 For $\forall \epsilon>0$, thus $\lim\sup P(M_n)\leq P(\sup_{0\leq y\leq 1}\vert B(y)\vert\geq c)$ On the other hand, for $\forall x,\ F_Z(x)\in[\delta_n,1-\delta_n]$,
 \begin{equation} \sup_{\delta_n\leq F_Z(x)\leq 1-\delta_n}\sqrt{n}\frac{f_Z(x)}{h^{'}(x)}\vert \widehat{g}(x)-h(x)\vert\geq \vert\frac{K(F_Z(x),n)}{\sqrt{n}}\vert-\vert \frac{K(F_Z(x),n)}{\sqrt{n}}-\sqrt{n}\frac{f_Z(x)}{h^{'}(x)}(\widehat{g}(x)-h(x))\vert
 \end{equation} Also using equation \ref{delta_n}, for $n$ sufficiently large, we have
 \begin{equation}
 \lim_{n\to\infty}\inf P(M_n>c)\geq \lim\inf P(\sup_{\delta_n\leq y\leq 1-\delta_n}\vert\frac{K(y,n)}{\sqrt{n}}\vert>c+\epsilon)=P(\sup_{0<y< 1}\vert\ B(y)\vert>c+\epsilon)
 \end{equation}
 When $y=0,1$ $B(y)=0\ a.s.$ since $B$ is a Brownian bridge, $c>0$, thus $P(\sup_{0<y< 1}\vert\ B(y)\vert>c+\epsilon)=P(\sup_{0\leq y\leq 1}\vert\ B(y)\vert>c+\epsilon)$. Also, since $\epsilon$ is arbitrary, we have $\lim_{n\to\infty}\inf P(M_n)\geq P(\sup_{0\leq y\leq 1}\vert\ B(y)\vert>c)$. From observation in theorem \ref{observation}, we know that, when $c>0$, $P(\sup_{0\leq y\leq 1}\vert B(y)\vert\geq c)=P(\sup_{0\leq y\leq 1}\vert\ B(y)\vert>c)$ and the theorem is proved.
 \end{proof}
 Now we will consider the alternatives, the next theorem shows that, if $h$ is sufficiently close to $g$ in the uniform norm, then the power of test will decrease.
 \begin{theorem}
 Consider the same test and same condition on theorem \ref{Test_Null}, $h$ is continuous differentiable on $[a,b]$ and has positive derivative on $(a,b)$.

 1) If $\exists x_0\in(a,b)$ such that $g(x_0)\neq h(x_0)$, then $P(M_n)\to 1$ as $n\to\infty$.

 2) We suppose alternative $H_1^{'}: h(x)=g(x)+\frac{1}{\sqrt{n}}s(x)$, here $s(x)\in C^{1}$ on $[a,b]$ and its derivative are no less than 0. Suppose $B(y)$ is a standard Brownian bridge, then the power of test satisfies
 \begin{equation}
 \lim_{n\to\infty}\sup P(M_n)\leq P(\sup_{0\leq y\leq 1}\vert B(y)\vert\geq c-\sup_{a<x<b}\frac{f_Z(x)\vert s(x)\vert}{g^{'}(x)})
 \end{equation}
 \end{theorem} \begin{proof} For sufficiently large $n$, $\delta_n<F_Z(x_0)<1-\delta_n$ since $F_Z$ is strictly increasing and $x_0\in(a,b)$. Then, we have \begin{equation} \begin{aligned} \sup_{\delta_n\leq F_Z(x)\leq 1-\delta_n}\sqrt{n}\frac{f_Z(x)}{h^{'}(x)}\vert \widehat{g}(x)-h(x)\vert\geq \sup_{\delta_n\leq F_Z(x)\leq 1-\delta_n}\sqrt{n}\frac{f_Z(x)}{h^{'}(x)}(\vert g(x)-h(x)\vert-\vert\widehat{g}(x)-g(x)\vert)\\ \geq \sqrt{n}\frac{f_Z(x_0)}{h^{'}(x_0)}\vert g(x_0)-h(x_0)\vert-\sqrt{n}\frac{f_Z(x_0)}{h^{'}(x_0)}\vert \widehat{g}(x_0)-g(x_0)\vert \end{aligned} \end{equation} Thus, \begin{equation} P(M_n)\geq P(\sqrt{n}\frac{f_Z(x_0)}{h^{'}(x_0)}\vert g(x_0)-h(x_0)\vert-c>\sqrt{n}\frac{g^{'}(x_0)}{h^{'}(x_0)}\frac{f_Z(x_0)}{g^{'}(x_0)}\vert \widehat{g}(x_0)-g(x_0)\vert) \label{pow} \end{equation} Since $\sqrt{n}\frac{f_Z(x_0)}{g^{'}(x_0)}(\widehat{g}(x_0)-g(x_0))\to_{a.s.}K(F_Z(x_0),n)/\sqrt{n}$ so it is $O_p(1)$. On the other hand, since $\sqrt{n}\vert g(x_0)-h(x_0)\vert\to\infty$, we know that $P(M_n)\to 1$ and the first part is proved.

 For the second part, notice that \begin{equation} \begin{aligned} \sup_{\delta_n\leq F_Z(x)\leq 1-\delta_n}\sqrt{n}\frac{f_Z(x)}{g^{'}(x)(1+\frac{s^{'}(x)}{\sqrt{n}g^{'}(x)})}\vert\widehat{g}(x)-g(x)-\frac{s(x)}{\sqrt{n}}\vert\\ \leq \sup_{\delta_n\leq F_Z(x)\leq 1-\delta_n}\sqrt{n}\frac{f_Z(x)}{g^{'}(x)}\vert\widehat{g}(x)-g(x)\vert+\sup_{a<x<b}\frac{f_Z(x)\vert s(x)\vert}{g^{'}(x)}
 \end{aligned}
 \end{equation}
 And we get the result.
 \end{proof}
 From this theorem, we know that, if $\sup_{a<x<b}\frac{f_Z(x)\vert s(x)\vert}{g^{'}(x)}$ is bigger than $c$, then it is possible for the power of test to get close to 1 as sample size is large. On the contrary, if this term is smaller than $c$, then the power of test will be less than 1 asymptotic even in the best situation.
 \section{Estimation for dependent data}
 In this section, we discuss estimation of transfer function under condition A2 or A3. We will firstly discuss convergence and uniform convergence of estimators of transfer function, and then we will apply re-sampling methods and construct confidence intervals.
 \begin{theorem} Assume A2 and B1, B2, then for given $x$, if $f_Z(x)>0$, then we have $\widehat{g}(x)\to_{a.s.} g(x)$. Moreover, if instead of A2, we assume $A2^{*}$: $Z_k=\Sigma_{i=0}^\infty a_i\epsilon_{k-i}$ and  $\sup_{x\in\mathbf{R}}\vert f_\epsilon^{''}(x)\vert<\infty$, the coefficients $a_i$ satisfy \label{Thm_5}
 \begin{equation}
 \Sigma_{i=1}^{\infty}\vert a_i\vert^{\min(\alpha/q,1)}<\infty
 \end{equation}
 for some $q\geq 2$ and $\alpha$ is the same as A2. Then, suppose $[c,d]$ being interval such that $\inf_{c\leq x\leq d}f_Z(x)>0$, then we have $\widehat{g}(x)\to g(x)$ almost surely and uniformly on $[c,d]$.
 \end{theorem}
 \begin{proof}
 According to theorem 1 in \cite{wu2005}, we choose $p$ in that theorem as $F_Z(x)$, since $f_Z(x)>0$, $\xi^Z(F_Z(x))=x$. This is because, on one hand, from definition of $\xi^Z$, since $F_Z(x)\geq F_Z(x)$, $x\geq \xi^Z(F_Z(x))$. On the other hand, for $y<x$ close to $x$, $F_Z(y)<F_Z(x)-\frac{1}{2}f_Z(x)(x-y)<F_Z(x)$. Since $F_Z$ is increasing and right continuous, from definition of $\xi^Z$, $F_Z(\xi^Z(F_Z(x)))\geq F_Z(x)\Rightarrow \xi^Z(F_Z(x))\geq x$, thus the equality holds. We have
 \begin{equation}
 \xi^Z_n(F_Z(x))-x=\frac{F_Z(x)-F_n^Z(x)}{f_Z(x)}+O_{a.s.}(n^{-3/4}(\log\log n)^{1/2}l_q^{1/2}(n))
 \end{equation} Here $l_q(n)=(\log\log n)^{1/2}$. According to \cite{wu2005}, we have that $\exists \sigma_1\in\mathbf{R}$ being a constant such that \begin{equation}
 \lim\sup_{n\to\infty}\pm\frac{\sqrt{n}(F_n^Z(x)-F_Z(x))}{\sqrt{2\log\log n}}=\sigma_1\Rightarrow\ F_n^Z(x)-F_Z(x)=O_{a.s.}\left(\frac{\sqrt{\log\log n}}{\sqrt{n}}\right)
 \end{equation} Thus, in particular, $\xi_n^Z(F_Z(x))\to_{a.s.} x$. Since $\xi^Y_n(F_Z(x))=g(\xi_n^Z(F_Z(x)))$ and $g$ is continuous, we have \begin{equation}
 \widehat{g}(x)=\xi^Y_n(F_Z(x))\to_{a.s.} g(x)
 \end{equation} For the second part, according to \cite{wu2005}, under the condition stated above, notice that $f_Z(x)>0,\ x\in[c,d]\Rightarrow F_Z(x)$ being strictly increasing and thus, \begin{equation} \sup_{c\leq x\leq d}\vert\xi_n^Z(F_Z(x))-x\vert=o_{a.s.}\left(\frac{c_q(n)}{\sqrt{n}}\right) \end{equation} Here, $c_q(n)=(\log n)^{1/q}(\log\log n)^{2/q}$ if $q>2$ and $(\log n)^{3/2}(\log\log n)$ if $q=2$. Since $[c,d]$ is closed interval and $g$ is continuous, thus is uniform continuous on $[c,d]$. Therefore, uniformly convergence is proved. \end{proof} Finally, we will consider construction of point-wise confidence intervals of transfer function. Here we will apply resampling methods to this problem. Here is a theorem dealing with this problem.
 \begin{theorem}
 Suppose B1, B2, A3 and the $A2^{*}$, and suppose $x$ is a given constant such that $\exists c<x<d$ and $\inf_{c\leq y\leq d}f_Z(y)>0$, $g^{'}(x)>0$. Define $\eta$ being a positive constant. For $b=b(n)$ satisfying: $b/n\to 0$ and $b\to\infty$, we define statistics
 \begin{equation}
 S_{n,b}(\eta,x)=\frac{1}{n-b+1}\Sigma_{i=1}^{n-b+1}\mathbf{1}\left\{\sqrt{b}\vert\widehat{g}_{b,i}(x)-\widehat{g}(x)\vert\leq \eta\right\} \end{equation}
 Here, $\widehat{g}_{b,i}(x)=\xi_{b,i}^Y(F_Z(x))$ with $\xi_{b,i}^Y(p)$ being sample quantile with sample $\left\{Y_i,Y_i+1,...,Y_i+b-1\right\}$ Then, we have: 1) $S_{n,b}(\eta,x)\to P(\sqrt{n}\vert\widehat{g}(x)-g(x)\vert\leq \eta)$ in probability. 2) Suppose $d(1-\alpha)=\inf\left\{\eta|S_{n,b}(\eta,x)\geq 1-\alpha\right\}$, then \begin{equation} P(\sqrt{n}\vert\widehat{g}(x)-g(x)\vert\leq d(1-\alpha))\to 1-\alpha \end{equation}
 \label{The_6}
 \end{theorem}
 \begin{proof}
 According to \cite{subsampling}, the only thing to prove is that $\sqrt{n}(\widehat{g}(x)-g(x))$ converges to a non-degenerated distribution. According to \cite{wu2005}, since for $q\geq 2$, for sufficiently large $n$, $\vert a_n\vert<1$ (otherwise the summation will not converge), then \begin{equation} \Sigma_{i=n}^{\infty}\vert a_i\vert^{\min(1,\alpha/2)}\leq \Sigma_{i=n}^{\infty}\vert a_i\vert^{\min(1,\alpha/q)}<\infty \end{equation} Thus, we have
 \begin{equation}
 \sqrt{n}(F_n^Z(x)-F_Z(x))\to N(0,\sigma^2)
 \end{equation} weakly. $N$ is a normal distribution with unknown variance. Therefore, according to \cite{wu2005}, similar with theorem \ref{Thm_5}, we have \begin{equation} \xi_n^Z(F_Z(x))-x=\frac{F_Z(x)-F_n^Z(x)}{f_Z(x)}+O_{a.s.}(n^{-3/4}(c_q(n)\log n)^{1/2}) \end{equation} Thus, $\sqrt{n}(\xi_n^Z(F_Z(x))-x)\to N(0,\sigma^2/f_Z^2(x))$. Since $g$ is differentiable at $x$, according to lemma \ref{lm2} and delta method, we have \begin{equation} \sqrt{n}(\widehat{g}(x)-g(x))=\sqrt{n}((g(\xi_n^Z(F_Z(x))))-g(x))\to N(0,g^{'}(x)^2\sigma^2/f_Z^2(x)) \end{equation} The result is proved. \end{proof}
\section{Numerical Experiments and Examples}
In this section, we discuss finite sample behaviors of the aforementioned statistics. We divide this section into two parts. In first part, we apply this statistics to several constructed data. In the second part, we will apply the aforementioned theory to study how well the primary settler of a urban waste water treatment plant clean the organics in the waste water (detail explanation and data can be gathered at\cite{Dua:2017} and the reference therein).
\subsection{Finite sample behavior of statistics on constructed data}
\begin{example}[i.i.d data with normal distribution] Here, we suppose $Z_i,\ i=1,2,...,n$ satisfy standard normal distribution. Notice that, for large $x$,
\begin{equation}
\begin{aligned}
\frac{1-F_Z(x)}{f_Z(x)}=\int_{x}^\infty\exp\left(\frac{1}{2}x^2-\frac{1}{2}t^2\right)dt=\int_{0}^\infty \exp\left(-\frac{1}{2}y^2-xy\right)\\
\leq\int_{0}^\infty \exp(-xy)dy=\frac{1}{x}
\end{aligned}
\end{equation}
Similarly, for $x\to-\infty$, $\frac{F_Z(x)}{f_Z(x)}=O\left(\frac{1}{\vert x\vert}\right)$. Notice that $\vert f_Z^{'}(x)/f_Z(x)\vert=x$. We constraint $x\in[-2,2]$, and choose $g(x)$ as 1) $(x+4)^2$, 2) $\log(x+5)$, 3) $x^3$. Notice that, for $g(x)=x^3$, it has 0 derivative at $x=0$ and we demonstrate how will the confidence band be influenced when assumption B3 is violated. Other functions all satisfy assumption B3. From example, we notice that, when derivative of $g(x)$ is not close to 0, confidence bands will be tight and close to confidence intervals, and when the derivative of $g(x)$ is relatively small, the performance of confidence bands will be inferior. When assumption B3 is violated, width of confidence bands will be severely influenced. The width of confidence intervals is not sensitive for small $g^{'}(x)$. However, large derivative of $g$ will affect width of confidence intervals.
\begin{figure}[htbp]
  \centering
  \includegraphics[width=12cm]{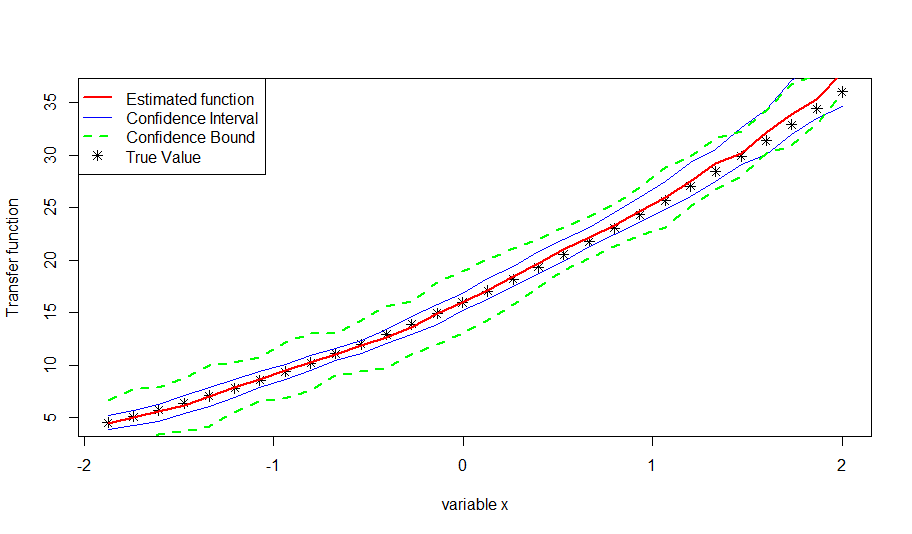}\\
  \caption{Estimator of $g(x)=(x+4)^2$, sample size is 1000 and confidence level is 0.99}
  \label{x_3_plus}
\end{figure}

\begin{figure}
\centering
\includegraphics[width=12cm]{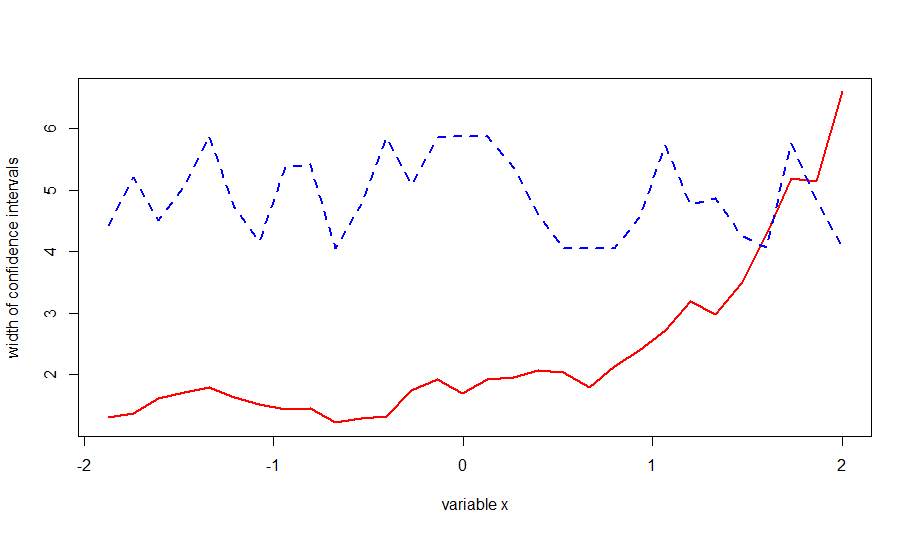}\\
\caption{Width of confidence intervals and bands of $g(x)=(x+4)^2$, solid line for confidence intervals and dashed line for confidence band}
\end{figure}

\begin{figure}[htbp]
  \centering
  \includegraphics[width=12cm]{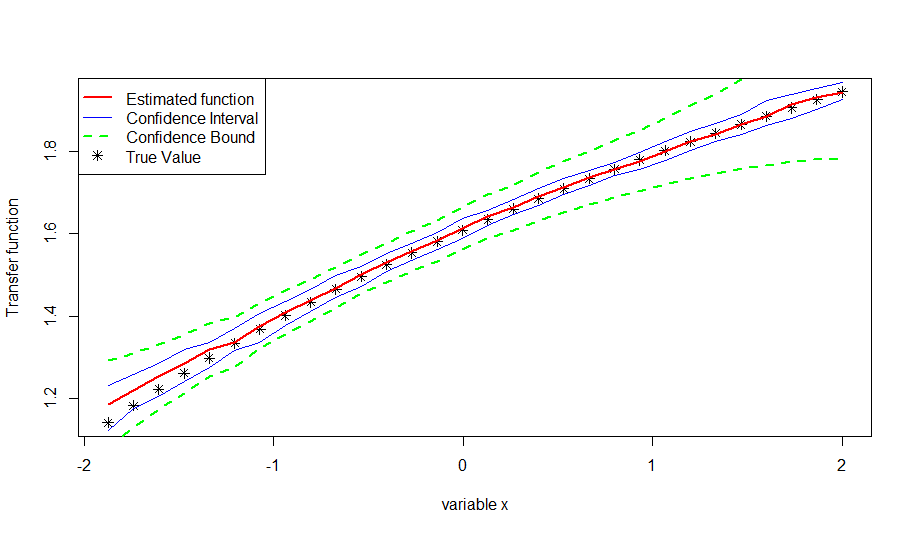}\\
  \caption{Estimator of $g(x)=\log(x+5)$, sample size is 1000 and confidence level is 0.99}
  \label{Log_x}
\end{figure}

\begin{figure}
  \centering
  \includegraphics[width=12cm]{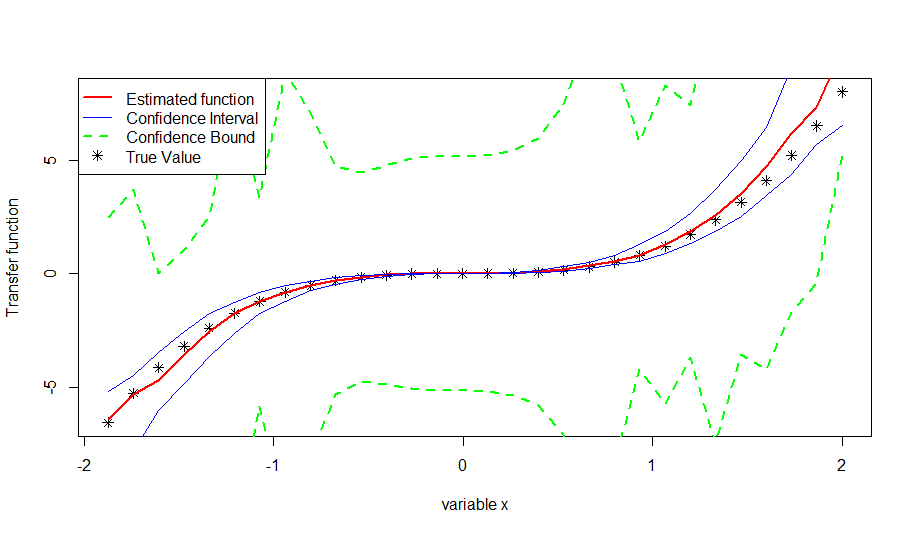}\\
  \caption{Estimator of $g(x)=x^3$, sample size is 1000 and confidence level is 0.99}\label{small_x3}
\end{figure}

\begin{figure}[htbp]
  \centering
  \includegraphics[width=12cm]{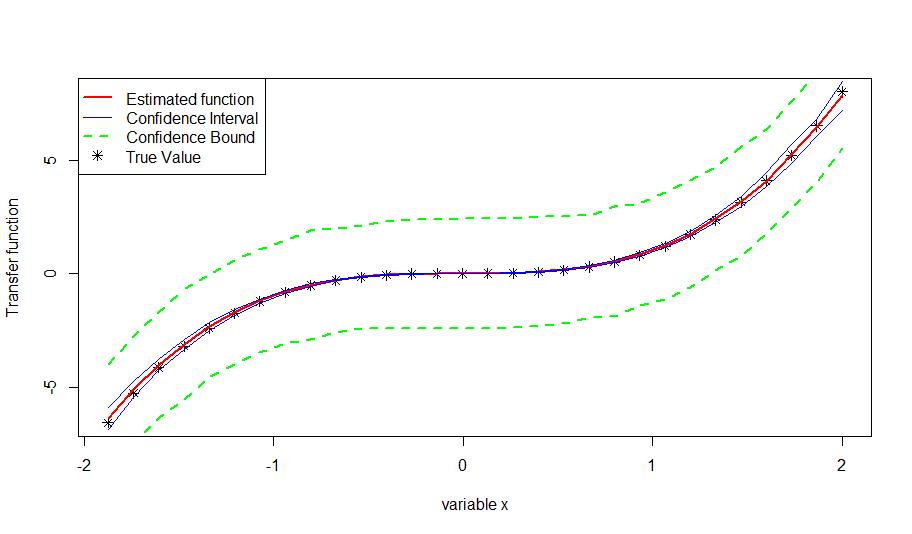}
  \caption{Estimator of $g(x)=x^3$, sample size is 20000 and confidence level is 0.99}\label{True_x_3}
\end{figure}
\label{Ex_1}
\end{example}

\begin{example}[Testing for equivalence of transfer function]
In this example, we examine finite sample performance of test under different $g(x)$ and different perturbation. We suppose sample size is $n_1=1000,\ n_2=10000$ and $h(x)=$ 1) $(x+4)^2$, 2) $\log(x+5)$ and 3) $e^x$. Also, we suppose the underlying $g(x)$ satisfies: 1) $g(x)=h(x)$, 2) $g(x)=h(x)+\frac{x}{n^{1/8}}$, 3) $g(x)=h(x)+\frac{x}{\sqrt{n}}$. We suppose $H_0: g(x)=h(x)$, perform test for 200 times and calculate the ratio of correct tests (that is, for assumption 1, test should accept $H_0$ to avoid first kind error and for assumption 2 and 3, test should reject $H_0$ to avoid second kind error). Confidence level is set as 0.85. The result is demonstrated in table \ref{Test_Sample_Size_1000} and \ref{Test_sample_Size_10000}. From the experiment, when difference of $h(x)$ and underlying $g(x)$ is of $O\left(\frac{1}{\sqrt{n}}\right)$, whether or not the test can separate $h$ and $g$ depends on the form of perturbation and is not strongly related to sample size.
\label{examp5}
\begin{table}
  \centering
  \caption{Ratio of correct test (definition see example \ref{examp5}) under different $g(x)$ and $h(x)$, sample size is 1000, confidence level is 0.85}
  \label{Test_Sample_Size_1000}
  \begin{tabular}{l| l l l}
  \hline\hline
  $h(x)$ & $g(x)=h(x)$ & $g(x)=h(x)+x/n^{1/8}$ & $g(x)=h(x)+x/\sqrt{n}$\\
  \hline
  $(x+4)^2$ & 0.84 & 0.32 & 0.165\\
  \hline
  $\log(x+5)$ & 0.87 & 1.0 & 0.985\\
  \hline
  $e^x$ & 0.91 & 1.0 & 0.46\\
  \hline\hline
  \end{tabular}
\end{table}

\begin{table}
  \centering
  \caption{Ratio of correct test (definition see example \ref{examp5}) under different $g(x)$ and $h(x)$, sample size is 10000 and confidence level is 0.85}
  \label{Test_sample_Size_10000}
  \begin{tabular}{l| l l l}
  \hline\hline
  $h(x)$ & $g(x)=h(x)$ & $g(x)=h(x)+x/n^{1/8}$ & $g(x)=h(x)+x/\sqrt{n}$\\
  \hline
  $(x+4)^2$ & 0.89 & 0.97 & 0.17\\
  \hline
  $\log(x+5)$ & 0.885 & 1.0 & 0.99\\
  \hline
  $e^x$ & 0.895 & 1.0 & 0.49\\
  \hline\hline
  \end{tabular}
\end{table}
\end{example}

\begin{example}[Transfer function estimation with normal MA data]
In this example, we suppose that the $Z_i,\ i=1,2,...,n$ are MA(10) normal data. That is, suppose i.i.d data $\epsilon_i,\ i=1,2,...$ satisfy standard normal distribution $N(0,1)$ and $Z_i=\Sigma_{k=0}^{10}\alpha_{k}\epsilon_{i-k},\ \alpha_0=1$. Notice that, marginal distribution of $Z_i$ is normal distribution $N(0,\Sigma_{k=0}^{10}\alpha_k^2)$. Moving average(k) sequence is strong mixing (definition can be seen in \cite{Strong_Mixing}) since $Z_t$ and $Z_{t+k}$ is independent. Also, it is obvious that condition A2 is satisfied for MA(10) sequence with normal innovation. We will choose coefficients as $\alpha_k=0.90^k$. Similarly as example \ref{Ex_1} $g(x)$ is chosen as 1) $(x+4)^2$, 2)$\log(x+10)$, 3)$x^3$. We choose $b(n)$ in  theorem \ref{The_6} as $n^{4/5}$.
\begin{figure}
  \centering
  \includegraphics[width=12cm]{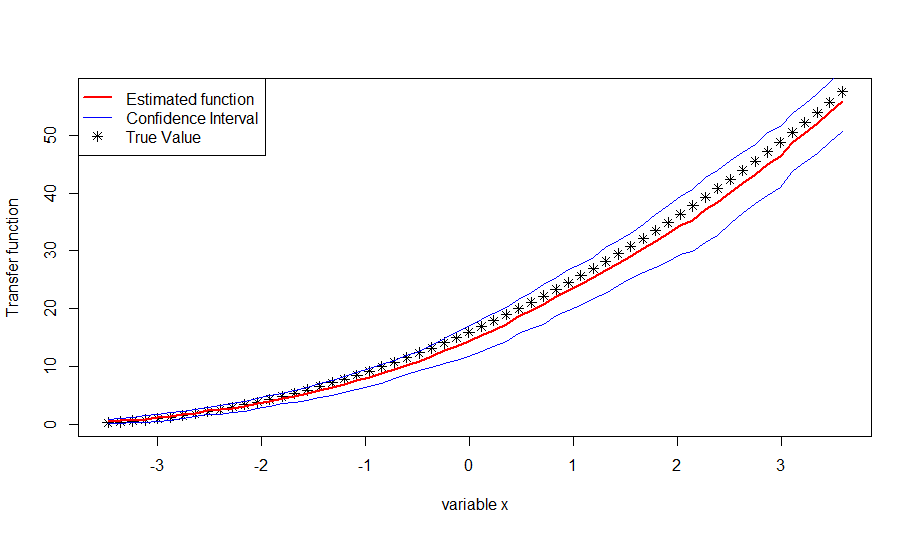}\\
  \caption{Estimator of $g(x)=(x+4)^2$ with dependent data. Sample size is 3000 and confidence level is 0.99}
  \label{Dependent_X_3}
\end{figure}

\begin{figure}
  \centering
  \includegraphics[width=12cm]{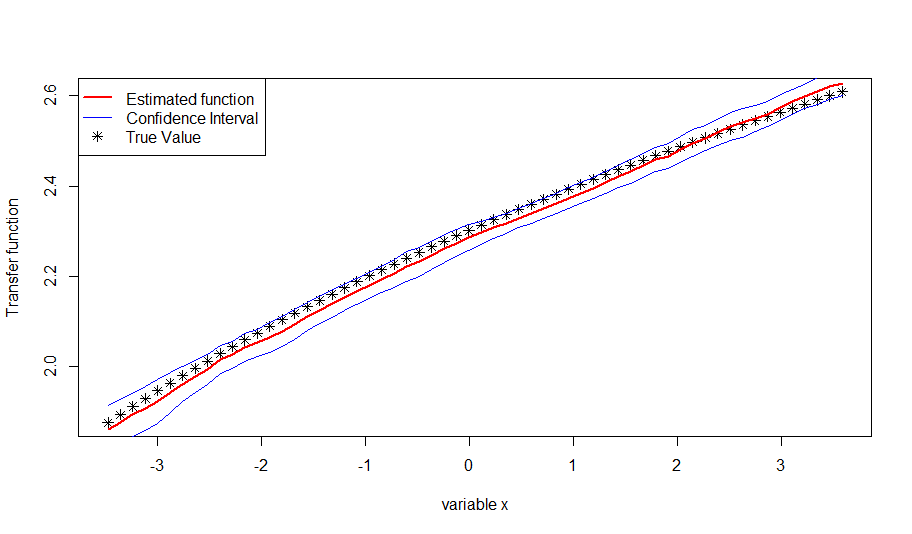}\\
  \caption{Estimator of $g(x)=\log(x+10)$ with dependent data. Sample size is 3000 and confidence level is 0.99}
  \label{Dep_log}
\end{figure}

\begin{figure}
  \centering
  \includegraphics[width=12cm]{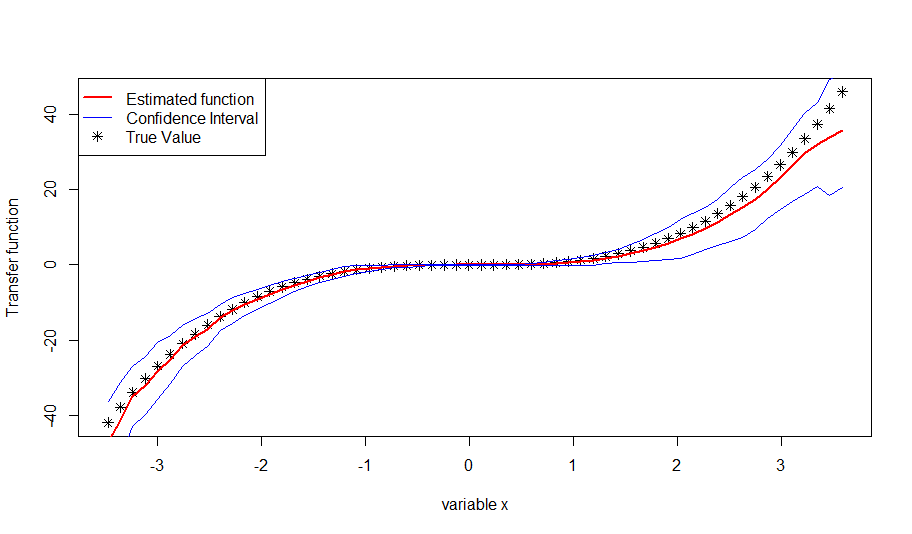}\\
  \caption{Estimator of $g(x)=x^3$ with dependent data. Sample size is 3000 and confidence level is 0.99}
  \label{X_3_dep}
\end{figure}
\end{example}
\subsection{Numerical study on water treatment plant data}
In this section, we will apply our estimator to study relationship between chemical demand of oxygen in input waste water (DQO-E) and the chemical demand of oxygen in water that has passed the primary settler (DQO-D) in a waste water treatment plant\cite{Dua:2017}. This index is always used to quantify amounts of organics in water. Instead of regression model, here we will treat DQO-E in wasted water as a random variable and suppose primary settler as a function $g$ that decreases the concentration of organics in the waste water. Thus, the remaining organics (quantified by DQO-D) is equal to $g(DQO-E)$. Intuitively, heavier the input water is polluted, more organics will be remained after the water being cleaned. Thus, it is safe to assume that $g$ is strictly increasing. Q-Q plot of gamma distribution and DQO-E shows that gamma distribution is a suitable approximation for DQO-E. Through maximum likelihood estimate, shape and scale parameter are estimated as 10.97 and 37.10, so we suppose that DQO-E has gamma distribution $\Gamma(10.97,0.0270)$. Notice that, gamma distribution with shape and rate $\alpha>1$ and $\beta$ has density $\frac{\beta^\alpha}{\Gamma(\alpha)}x^{\alpha-1}\exp(-\beta x)$. Thus, we have
\begin{equation}
\frac{f_Z^{'}(x)}{f_Z(x)}=\frac{\alpha-1}{x}-\beta
\end{equation}
 When $x\to 0$, $f_Z(x)>0$ and thus $F_Z(x)\leq xf_Z(x)$. Thus, as long as
\begin{equation}
\frac{g^{''}(x)x}{g^{'}(x)}=O\left(1\right),\ x\to 0
\label{Real_Cond_1}
\end{equation}
condition B3 is satisfied when $x\to 0$. On the other hand, notice that, as $x$ being large
\begin{equation}
\begin{aligned}
\frac{1-F_Z(x)}{f_Z(x)}=x\int_{0}^1(z+1)^{\alpha-1}\exp(-\beta xz)dz+x\int_{1}^\infty (z+1)^{\alpha-1}\exp(-\beta xz)dz\\
\leq x\int_{0}^1 2^{\alpha-1}\exp(-\beta xz)dz+2^{\alpha-1}\int_{0}^\infty\frac{z^{\alpha-1}}{x^{\alpha-1}}\exp(-\beta z)dz\\
=\frac{2^{\alpha-1}}{\beta}(1-\exp(-\beta x))+\frac{2^{\alpha-1}}{\beta^\alpha x^{\alpha-1}}\mathbf{\Gamma}(\alpha)
\end{aligned}
\end{equation}
Here, $\mathbf{\Gamma}(\alpha)$ is gamma function and since $\alpha>0$, gamma function converges absolutely. Thus, as long as 
\begin{equation}
\frac{g^{''}(x)}{g^{'}(x)}=O\left(1\right),\ x\to\infty
\label{Real_Cond_2}
\end{equation}
condition B3 is satisfied as $x\to\infty$. We suppose transfer function $g$ in the example satisfies condition \ref{Real_Cond_1} and \ref{Real_Cond_2}.

 We apply the test introduced in theorem \ref{Test_Null} to test whether gamma distribution suits DQO-E data or not(that is, suppose DQO-E is a function $h$ of a $\Gamma(10.97,0.0270)$ random variable and test $h(x)=x$). In order to avoid bias introduced by estimated shape and scale parameters, we use Monte Carlo method discussed by Julian and Peter \cite{doi:10.1093/biomet/76.4.633} to calculate p-value. The result is demonstrated in table \ref{Test_Nom}. Figure \ref{Real_Para} demonstrates the relations between DQO-E and DQO-D. Slope of $g$ will decrease as input demand of oxygen in waste water increases, so we can make conclusion that primary settler is efficient in cleaning organics when there is high concentration of organic matters in waste water.
\label{water_tre}
\begin{table}
  \centering
  \caption{Test for fitting gamma distribution of chemical demand of oxygen in input waste water}
  \label{Test_Nom}
  \begin{tabular}{l| l l}
  \hline\hline
  Null assumption & statistics & P-value\\
  \hline
  $h(x)=x$ &0.768 & 0.546\\
  \hline\hline
  \end{tabular}
\end{table}

\begin{figure}
  \centering
  \includegraphics[width=12cm]{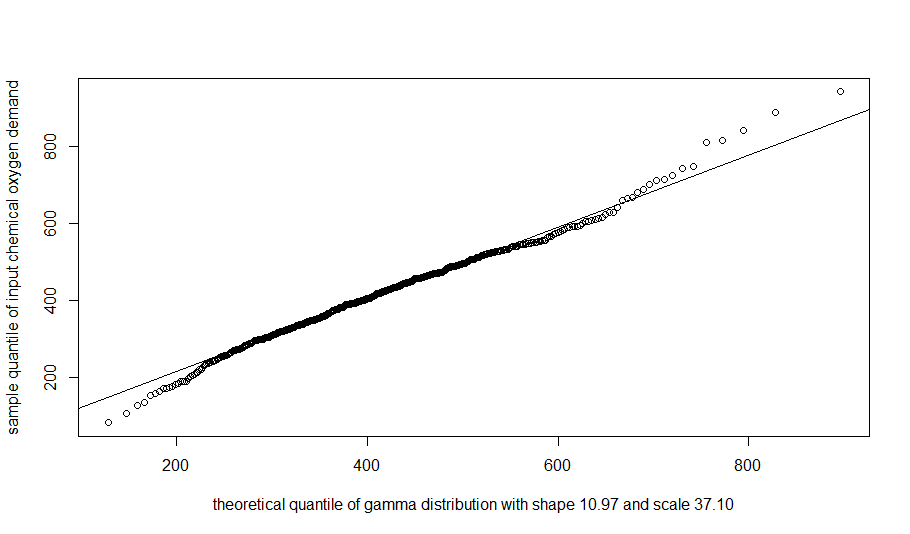}
  \caption{Q-Q plot for chemical demand of oxygen in input waste water}
  \label{QQ}
\end{figure}

\begin{figure}
  \centering
  \includegraphics[width=12cm]{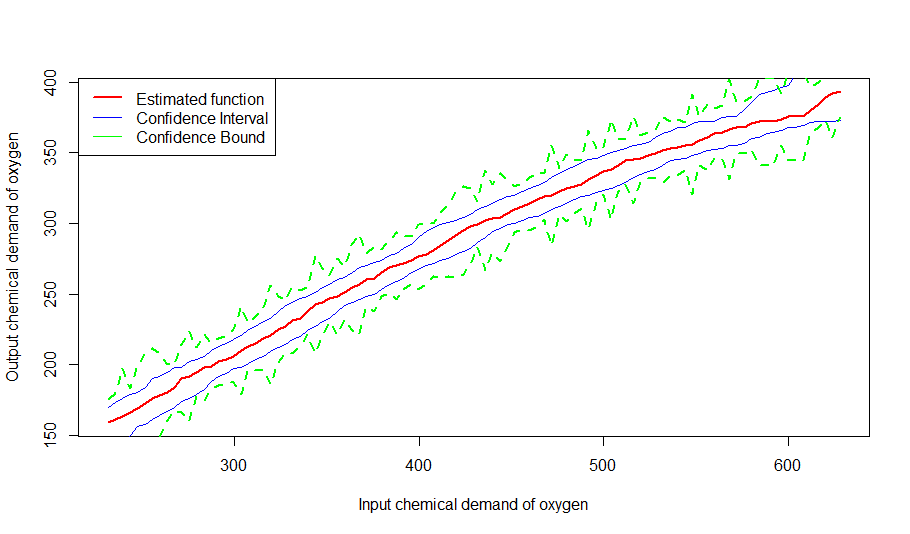}\\
  \caption{Relation between DQO-E and DQO-D (definition see \ref{water_tre}), sample size is 518 and confidence level is 0.99}
  \label{Real_Para}
\end{figure}

\section{Conclusion}
In this paper, we focus on model $Y_i=g(Z_i),\ i=1,2,...$ with $Z_i$ being random variables with known distribution and $g(x)$ being unknown strictly monotonic function. We try to estimate $g(x)$ in this model. For i.i.d data, we propose estimator of $g(x)$ and construct point-wise confidence intervals as well as confidence bands. For short-range dependent data, we prove the consistency of the proposed estimator and use a resampling method to create confidence intervals. Moreover, a goodness of fit test for correctness of $g(x)$ is presented and an alternative of this test is discussed as well.

In numerical part, we study finite sample performance of estimator and test for different $g(x)$ and alternatives. width of confidence bands are sensitive with $g^{'}(x)$. If $g^{'}(x)$ is close to 0, then confidence bands will be much wider than point-wise confidence intervals and if $g^{'}$ is relatively large, then confidence bounds will be close to confidence intervals. On the contrary, small derivative of $g$ will not severely affect point-wise confidence intervals.

In reality, this model can be applied to study relations between input signals with known distribution and responses with unknown distribution, such as correspondence between quality of materials and quality of products, electricity signals with white noises and power of motors, significance of a symptom and concentration of toxic materials in the atmosphere, etc.
\bibliographystyle{unsrt}
\bibliography{Reference_Note}
\end{document}